\documentclass{article}
\usepackage{amsmath,amssymb,amsthm}
\newtheorem{Thm}{Theorem}
\newtheorem{Lem}{Lemma}

\theoremstyle{definition}

\newtheorem{Rem}{Remark}

\begin{document}

\title{On semilinear elliptic equations with global coupling}
\author{Shinji Kawano  \\ Department of Mathematics \\ Hokkaido University \\ Sapporo  060-0810, Japan}
\date{}
\maketitle

\begin{abstract}

We consider the problem
\begin{equation*}
 \begin{cases}
   \triangle A- A+A^p-kA\int_{\mathbb{R}^n} A^2dx=0  & \text{in  $\mathbb{R}^n$},\\
   \displaystyle \lim_{\lvert x \rvert \to \infty} A(x)  =0, 
 \end{cases}  
\end{equation*}
where $p>1,k>0$ are constants. We classify the existence of all possible positive solutions to this problem.
\end{abstract}

\section{Introduction}
Equations with global coupling arise in the study of pattern formation in various fields of science such as fluid mechanics as well as
chemistry or biology. Especially Riecke~\cite{Riecke} has suggested the following equation:
\begin{equation}
A_t=\triangle A + \mu A + c \lvert A \rvert ^2 A- \lvert A \rvert ^4 A -k\int_{\mathbb{R}^n} \lvert A \rvert ^2 dx, \label{original}
 \end{equation}
where A is a complex-valued function defined on $\mathbb{R}^n \times (0,\infty)$, $k>0$ and $c,\mu$ are real numbers.

We are interested in positive solutions. This means that we replace \eqref{original} by the following equation:
\begin{equation}
A_t=\triangle A + \mu A + cA^3-A^5-kA\int_{\mathbb{R}^n} A^2 dx, \label{time}
\end{equation}
where A is a positive function defined on $\mathbb{R}^n \times (0,\infty)$.

Wei and Winter~\cite{WW} studied this eequation \eqref{time}, together with the steady-state solutions of \eqref{time} which satisfy
\begin{equation}
\triangle A+ \mu A + cA^3-A^5-kA\int_{\mathbb{R}^n} A^2 dx=0. \label{positive}
\end{equation}
They have classified the existence of positive solutions to \eqref{positive} and studied the stability of all possive standing wave solutions. For other researches for equations concerned with pattern formation, consult Matthews and Cox~\cite{NON1} and 
Norbury, Wei and Winter~\cite{NON2}.

In this present paper, we shall consider the following problem 
\begin{equation}
 \begin{cases}
   \triangle A- A+A^p-kA\displaystyle \int_{\mathbb{R}^n} A^2dx=0  & \text{in  $\mathbb{R}^n$},\\
   \displaystyle \lim_{\lvert x \rvert \to \infty} A(x)  =0,  \qquad  A(0)=\max_{x\in \mathbb{R}^n}A(x),     \label{a}
 \end{cases}  
\end{equation}
where A is a positive-valued function defined on $\mathbb{R}^n \times (0,\infty)$ and $p>1,k>0$ are constants, 
$n$ is the space dimension $n\ge 1$, integer.  

This equation is from \eqref{positive}, modified and generalized to seize mathematical structures.
We shall find classical positive solutions to \eqref{a}, using the uniqueness result to single power nonlinear equation given by Kwong~\cite{K}.

Following is the main result of the present paper:
\begin{Thm}  
For any $p\in (0,p^*(n))$, the problem \eqref{a} has a positive radially symmetric solution if $k>0$ is sufficiently small.

For any $p\in [p^*(n),\infty)$, the problem \eqref{a} does not have any solutions for any $k>0$.

Here, $p^*(n)$ is the exponent called Sobolev exponent defined by the following formula:
\begin{equation*}
p^*(n)=\begin{cases} 
         \infty            \qquad n=1,2.    \\
         \dfrac{n+2}{n-2}  \qquad  n\ge 3.
     \end{cases}
\end{equation*} 
\end{Thm}
In the next section we present the proof of the theorem and more detailed formulation of the theorem.
We note here that the method is essentially based on Wei and Winter~\cite{WW}.

\section{Proof of the theorem}
First we state the key lemma for the proof.
\begin{Lem}
Consider the problem 
\begin{equation}
 \begin{cases}
  \triangle A -\omega A + A^p=0,   \\
  \displaystyle \lim_{\lvert x \rvert \to \infty} A(x)  =0,  \qquad  A(0)=\max_{x\in \mathbb{R}^n}A(x),     \label{omega}
 \end{cases}
\end{equation}
where $p>1$ and $\omega>0$.

1. If~~$1<p<p^*(n)$,
the problem \eqref{omega} posesses the unique positive radially symmetric classical solution for each $\omega>0$.

Moreover, the solution $A$ is given by the following formula:
\begin{equation}
 A(x)=\omega^{\frac{1}{p-1}} A_0(\sqrt{\omega}x), \label{rep}
\end{equation}
where $A_0$ is the unique solution of the following problem:
\begin{equation}
 \begin{cases}
  \triangle A -A + A^p=0,   \\
  \displaystyle \lim_{\lvert x \rvert \to \infty} A(x)  =0,  \qquad  A(0)=\max_{x\in \mathbb{R}^n}A(x),  \label{1}
 \end{cases}
\end{equation}
(This problem \eqref{1} is the problem \eqref{omega} with $\omega=1$.)

2. If~~$p^*(n)\le p<\infty$
the problem \eqref{omega} does not have any positive solutions.

3. The function $A_0$ has exponential decay at infinity:
\begin{equation*}
\lvert A_0(x) \rvert \le Ce^{-\delta \lvert x \rvert}, \qquad x \in \mathbb{R}^n,
\end{equation*}
for some $C,\delta>0$.    \label{lem}
\end{Lem}
\begin{proof} 
The existence and non-existence part is due to the classical work of Pohozev~\cite{POV}.
The fact that for $p\ge p^*(n)$, \eqref{omega} does not have any solution follows from a well known identity of Pohozaev. 
Also see Berestycki and Lions~\cite{B1} and Berestycki, Lions and Peletier~\cite{B2}. 

The decay rate of $A_0$ is studied in Berestycki and Lions~\cite{B1}.

The classical work of Gidas, Ni and Nirenberg~\cite{G1,G2} tells us that all positive solutions are radially symmetric.

The uniqueness result of \eqref{1} is by Kwong~\cite{K} for $n\ge 3$, by Mcleod and Serrin~\cite{MS} for $n=2$, 
and by Berestycki and Lions~\cite{B1} for $n=1$. 
It is enough to prove uniqueness for \eqref{omega} with any $\omega >0$, and the representation formula~\eqref{rep}.

$A(x)$ defined by the formula \eqref{rep} is a solution to the problem \eqref{omega}.
On the other hand, for the solution $B(x)$ to \eqref{omega}, $B_0(x):=\omega^{-\frac{1}{p-1}}B(\frac{x}{\sqrt{\omega}})$ is a solution to \eqref{1} .
By the uniqueness theorem for \eqref{1}, this $B_0$ is equal to $A_0$ at any point.
This fact completes the proof.
\end{proof}

First we note that the problem \eqref{a} is equivalent to
\begin{equation*}    
 \begin{cases}
  \triangle A-\omega A+A^p=0,    \\
  \omega = 1+k\displaystyle \int_{\mathbb{R}^n} A^2 dx,   \\
  \displaystyle \lim_{\lvert x \rvert \to \infty} A(x)  =0,  \qquad  A(0)=\max_{x\in \mathbb{R}^n}A(x).
  \end{cases}
\end{equation*}

So the solution of \eqref{a} is of the form
\begin{equation*}
A(x)=\omega^{\frac{1}{p-1}} A_0(\sqrt{\omega}x),
\end{equation*}
with the consistency condition
\begin{align}
  \omega &= 1+k \int_{\mathbb{R}^n}\omega^{\frac{2}{p-1}} A_0(\sqrt{\omega}x)^2 dx  \notag  \\
         &= 1+k \int_{\mathbb{R}^n}\omega^{\frac{2}{p-1}-\frac{n}{2}} A_0(y)^2 dy.  \label{cons}
\end{align}
The relation \eqref{cons} is equivalent to
\begin{equation*}
k \alpha =(\omega -1) \omega^{\frac{n}{2}-\frac{2}{p-1}},
\end{equation*}
where 
\begin{equation*}
\alpha = \displaystyle \int_{\mathbb{R}^n} A_0(y)^2 dy.
\end{equation*}
Note that $0<\alpha<\infty$ from the fact 3. of the Lemma~\ref{lem}.

We analyze the function  
\begin{equation*}
f(\omega):=(\omega -1) \omega^{\frac{n}{2}-\frac{2}{p-1}}.
\end{equation*}
We define $e_{n,p}:=\dfrac{n}{2}-\dfrac{2}{p-1}$.
Five cases occur when we investigate the function $f(\omega)$ in $(0,\infty)$:
\begin{itemize}
\item[1]~~~$e_{n,p}+1<0$    \qquad  i.e.  $1<p<1+\dfrac{4}{n+2}$.                      
\begin{quotation}
$f(\omega)$ attains its maximum at $\omega_{n,p}:=\dfrac{-e_{n,p}}{-e_{n,p}-1}>0$, and the maximum is positive.

$f(+0):=\displaystyle \lim_{\omega \to 0}f(\omega)=-\infty$, \qquad $f(\infty):=\displaystyle \lim_{\omega \to \infty}f(\omega)=0$.
\end{quotation}
\item[2]~~~$e_{n,p}+1=0$      \qquad  i.e. $p=1+\dfrac{4}{n+2}$.                      
\begin{quotation}
$f(\omega)$ is increasing in $(0,\infty)$.  \qquad

$f(+0)=-\infty$,\qquad $f(\infty)=1$.
\end{quotation}
\item[3]~~~$0<e_{n,p}+1<1$    \qquad  i.e. $1+\dfrac{4}{n+2}<p<1+\dfrac{4}{n}$.     
\begin{quotation}
$f(\omega)$ is increasing in $(0,\infty)$.  \qquad

$f(+0)=-\infty$, \qquad $f(\infty)=\infty$.
\end{quotation}
\item[4]~~~$e_{n,p}+1=1$      \qquad  i.e. $p=1+\dfrac{4}{n}$.                       
\begin{quotation}
$f(\omega)$ is increasing in $(0,\infty)$.  \qquad

$f(+0)=-1$,\qquad $f(\infty)=\infty$.
\end{quotation}
\item[5]~~~$1<e_{n,p}+1$      \qquad  i.e. $1+\dfrac{4}{n}<p$.
\begin{quotation}
$f(\omega)$ attains its minimum at $\omega_{n,p}=\dfrac{e_{n,p}}{e_{n,p}+1}>0$, and the minimum is negative.

$f(+0)=0$, \qquad $f(\infty)=\infty$.
\end{quotation}
\end{itemize}
Following is the detailed statement of the theorem:
\begin{Thm}
Let the space dimension $n\ge 1$ be fixed.
\begin{itemize}
\item[1]~~~$1<p<1+\dfrac{4}{n+2}$.                      
\begin{itemize}
\item[1-1]~~~If~~$k>\dfrac{f(\omega_{n,p})}{\alpha}$, then the problem \eqref{a} does not have any solution.
\item[1-2]~~~If~~$k=\dfrac{f(\omega_{n,p})}{\alpha}$, then the problem \eqref{a} has exactly one solution.
\item[1-3]~~~If~~$\dfrac{f(\omega_{n,p})}{\alpha}>k>0$, then the problem \eqref{a} has two solutions.
\end{itemize}

\item[2]~~~$p=1+\dfrac{4}{n+2}$.                      
\begin{itemize}
\item[2-1]~~~If~~$k\ge 1$, then the problem \eqref{a} does not have any solution.
\item[2-2]~~~If~~$0<k<1$, then the problem \eqref{a} has exactly one solution.
\end{itemize}

\item[3]~~~$1+\dfrac{4}{n+2}<p<p^*(n)$.     
\begin{itemize}
\item[3-1]~~~For any $k>0$, the problem \eqref{a} has exactly one solution. 
\end{itemize}

\item[4]~~~$p^*(n)\le p$.
\begin{itemize}
\item[4-1]~~~For any $k>0$, the problem \eqref{a} does not have any solution. 
\end{itemize}
\end{itemize}      
Here a solution means a positive radial symmetric classical solution.  \label{Thm}
\end{Thm}
\begin{proof}[Proof of Theorem~\ref{Thm}]
For $p\ge p^*(n)$. Suppose there is a solution $A(x)$, for contradidtion.
Then $A$ is a solution of
\begin{equation}    
\triangle A-\omega A+A^p=0    \label{Shinji}
\end{equation}
 with 
\begin{equation}
\omega:= 1+k\displaystyle \int_{\mathbb{R}^n} A^2 dx.   \label{Kawano}
\end{equation}
This contradicts to the fact 2. of the Lemma~\ref{lem}.

From now on we concentrate on the case $1<p<p^*(n)$. 
\begin{Lem}
Suppose that for $k>0$, there exists $\omega>0$ such that $f(\omega)=k\alpha$. Then $A(x):=\omega^{\frac{1}{p-1}} A_0(\sqrt{\omega}x)$ is a solution to \eqref{a}.

On the other hand, for a solution $A(x)$ of \eqref{a} with $k>0$, there exists $\omega>0$ such that $f(\omega)=k\alpha$.
Moreover, $A$ is formulated with this $\omega$: $A(x)=\omega^{\frac{1}{p-1}} A_0(\sqrt{\omega}x)$.
\end{Lem}
\begin{proof}
First remember that for $A(x):=\omega^{\frac{1}{p-1}} A_0(\sqrt{\omega}x)$, 
$f(\omega)=k\alpha$
is equivalent to
$1+ \displaystyle \int_{\mathbb{R}^n} A^2dx = \omega$.

Now we prove the first statement.
The left hand side of the equation in problem \eqref{a} is calculated in the following way: 
\begin{align*}
  \triangle A- A+A^p-kA\displaystyle \int_{\mathbb{R}^n} A^2dx  
  &=  \triangle A- \left(1+ \displaystyle \int_{\mathbb{R}^n} A^2dx   \right) A+A^p  \\
  &=  \triangle A -\omega A + A^p  \\
  &= 0.
\end{align*}
Then we prove the last statement. A solution $A$ of \eqref{a} is a solution of \eqref{Shinji} with \eqref{Kawano}.
Here $\omega$ has to be positive, for otherwise the equation \eqref{Shinji} cannot have any solution.
From \eqref{Shinji} comes the formula $A(x)=\omega^{\frac{1}{p-1}} A_0(\sqrt{\omega}x)$,  
and the relation $f(\omega)=k\alpha$ is from \eqref{Kawano}.
\end{proof}
The above lemma and the observation of the graph of $f(u)$ completes the proof of Theorem~\ref{Thm}.
\end{proof}
\begin{Rem}
Note that any solution of our problem \eqref{a} is in the form \eqref{rep}: these solutions converges to $A_0$ pointwise as $k$ in \eqref{a} goes to zero.
\end{Rem}
\begin{Rem}
We can treat a more generalized problem in the same way:
\begin{equation*}
 \begin{cases}
  \triangle A- A+A^p-kA\displaystyle \int_{\mathbb{R}^n} A^rdx=0 \qquad  \text{in  $\mathbb{R}^n$},   \\
  \displaystyle \lim_{\lvert x \rvert \to \infty} A(x)  =0,  \qquad  A(0)=\max_{x\in \mathbb{R}^n}A(x),    
 \end{cases}
\end{equation*}
where r is positive. We need three cases to be dealt with if $n\ge 3$:
\begin{align*}
0<r<&\dfrac{2n}{n-2}, &
\dfrac{2n}{n-2}\le r&<\dfrac{2(n+2)}{n-2}, &
\dfrac{2(n+2)}{n-2}\le r. 
\end{align*}
Our $r=2$ is always in the first case, which is most bothersome. We shall omit the full observation. 
\end{Rem}

\end{document}